\documentclass[11pt]{article}

\usepackage{latexsym}
\usepackage{amssymb,amsmath}

\usepackage{hyperref}

\usepackage[active]{srcltx}

\usepackage{graphicx}
\usepackage{amscd}

\usepackage{amsfonts}
\usepackage{mathrsfs}
\usepackage{esint}

\newtheorem{theorem}{Theorem}[section]
\newtheorem{lemma}[theorem]{Lemma}
\newtheorem{corollary}[theorem]{Corollary}
\newtheorem{definition}[theorem]{Definition}
\newtheorem{example}[theorem]{Example}

\newtheorem{remark}[theorem]{Remark}

\numberwithin{equation}{section}
\DeclareMathOperator*{\ei}{ess\inf}

\newcommand{\bp}{\noindent{\bf Proof. }}
\newcommand{\ep}{\hfill$\Box$\medskip}

\font\doppio=msbm10 at 12pt
\font\corsivo=rsfs10 at 12pt

\newcommand{\RR}{\mathbb R}
\newcommand{\G}{\mathbb G}

\newcommand{\A}{\mathscr A}
\newcommand{\hei}{\hbox{\doppio H}}     
\newcommand{\mint}{{\int\!\!\!\!\!\!-}}

\newcommand{\diver}{\mathrm{div}} 

\newcommand{\C}{\mbox{\corsivo C}}

\newcommand{\Cuno}{\mbox{\corsivo C}^{\,1}}

\newcommand{\RN}{\RR^N}

\newcommand{\abs}[1]{\left| #1 \right|} 
\newcommand{\norm}[1]{\left|\left| #1 \right|\right|} 
\newcommand{\plap}[1]{\diverl(\abs{\grl #1}^{p-2}\grl #1) }
\newcommand{\qlap}[1]{\diverl(\abs{\grl #1}^{q-2}\grl #1) }
\newcommand{\plapE}[1]{\diver(\abs{\nabla #1}^{p-2}\nabla #1) }
\newcommand{\qlapE}[1]{\diver(\abs{\nabla #1}^{q-2}\nabla #1) }

\newcommand{\grh}{\nabla_{\!\!H}} 
\newcommand{\lh}{\Delta_H} 

\newcommand{\grl}{\nabla_{\!\!L}} 
\newcommand{\diverl}{\diver_{\!L}} 
\newcommand{\decl}{:=}

\title{ Entire solutions of quasilinear elliptic systems\\
on Carnot Groups}

\author{Lorenzo D'Ambrosio 
\\
{\small  Dipartimento di Matematica,
Universit\`a degli Studi di Bari}\\
{\small via E. Orabona, 4,  I-70125 Bari, Italy,
 {\tt dambros@dm.uniba.it}} \\ \\
 Enzo Mitidieri\thanks{Corresponding author, e-mail {\tt mitidier@units.it.}}
\\
{\small Dipartimento di Matematica e Geoscienze,
Universit\`a degli Studi di Trieste}\\
{\small via A.Valerio, 12/1, I-34127 Trieste, Italy,
 {\tt mitidier@units.it  }}
\\
\\ {\it
{To the memory of 
Professor Lev Dmitrievich Kudryavtsev}}}

\date{September 15, 2012}
\begin{document}

\maketitle

\begin{abstract} We prove general a priori estimates of solutions  of  a class of quasilinear elliptic system on Carnot groups.
As a consequence, we obtain several non--existence theorems. 
The results are new even in the Euclidean setting.
\end{abstract}

\noindent{\footnotesize{\it Keywords:} Quasilinear elliptic systems; {\it A priori} estimates; Non--existence theorems; Positive solutions; Carnot groups.}




\section{Introduction}\label{sec1}
As it is well known, one of the main problems in the theory of nonlinear partial differential equations is to find a priori bounds on the possible solutions of the problem under consideration. This information is crucial from several point of view. 

 On one hand the bounds that one can prove may be used for improved regularity properties of the solutions and on the other hand
these results are crucial for establishing special qualitative properties of them. For a recent contribution in this direction see D'Ambrosio, Farina, 
Mitidieri and Serrin \cite{DFM} and D'Ambrosio and Mitidieri \cite{dmginz}.

In this paper we consider a class of quasilinear elliptic systems on Carnot groups and  prove general {\it a priori} estimates of positive solutions  in an open set of $\mathbb R^N$.

There are several  recent studies dealing with this problem in the Euclidean framework. See for instance \cite{Gidas, mp1, cfm, mp, bvp, SZ, dam09, dm}.

To our knowledge, this is the first attemp to prove general estimates of solutions of quasilinear systems on structures which are not necessarily Euclidean.

Among other possibilities, this allows to extend known existence results related to the classical Dirichlet problem in Euclidean setting to the Carnot framework by using topological methods via blow-up procedure  in the same spirit of \cite{gs, cfm, ACM, zou}.

In this paper we prove {\it a priori} estimates for the solutions of elliptic systems in an open set $\Omega\subseteq\RR^N$ involving quasilinear operators in divergence form. As a consequence, we obtain some nonexistence results for these problems in all of $\RR^N$. 

 Earlier contributions on the {\it nonexistence question}  for semilinear {\it  scalar}  subelliptic problems with power nonlinearities  were obtained by Capuzzo-Dolcetta and Cutr{\`{\i}} \cite{cap}, Birindelli, Capuzzo-Dolcetta and  Cutr{\`{\i}} \cite{BIR1}. The quasilinear case was studied by D'Ambrosio \cite{dam09}. 
More recently, for general nonlinearties, the quasilinear {\it scalar  case}  has been studied in D'Ambrosio and Mitidieri \cite{dm}, \cite{dmsb} and \cite{dmginz}. 
 
 The results proved in this paper are new even in the Euclidean setting.

To be more precise our aim is to study problems of the type,
$$
\begin{cases}-\mathrm{div}(\mathscr A_p(x,u,\nabla u))\ge f(x,u,v) \quad&\mbox{on }\Omega,\\
\\
-\mathrm{div}(\mathscr A_q(x,v,\nabla v))\ge g(x,u,v) \quad&\mbox{on }\Omega,\\
\\
u\ge0, \; v\ge0\quad&\mbox{on } \Omega,
\end{cases}\eqno{(P)}
$$
where, 
$$\mathscr A_p, \mathscr A_q:\Omega\times\mathbb R\times\mathbb R^N\to\mathbb R^N$$ 
are strongly-$p$-coercive and strongly-$q$-coercive ($p,q>1$) respectively, 
and
$$f, g:\Omega\times[0,\infty)\times [0,\infty)\to[0,\infty)$$ are Carath\'edory functions. 
On the possible solution $(u,v)$ of (P), we do not require 
any kind of behavior near the boundary of $\Omega$ or at infinity.

Throughout this work, we shall  essentially use the same  ideas as in \cite{dm}, 
where we deal with general  estimates of solutions of scalar differential inequalities.

One of the typical result proved in this paper in the Euclidean setting is the following.
Let $\Omega=\RR^N,$ $f(x,u,v)=f(v)$ and $g(x,u,v)=g(u)$.
Suppose  that the following local assumptions on  the  {\it behavior near zero}
 of $f$ and $g$ hold:
 $$ \liminf_{t\to 0}\frac{f(t)}{t^{{a}}}>0\quad (possibly\ +\infty), \ \ with \quad {a}>0,\eqno{(f_0)}$$
  $$\liminf_{t\to 0}\frac{g(t)}{t^{{b}}}>0 \quad (possibly \ +\infty),\quad with\ b>0. \eqno{(g_0)}$$

{\it Let $(u,v)$ be a weak solution of (P) such that $\ei_{\mathbb R^N}u=\ei_{\mathbb R^N}v=0$. 
If
  \begin{equation}    \label{eq:hyp0}
    \min\left\{ N-{p}-({p}-1)\frac{{q}}{{b}},\ \  N-{q}-({q}-1)\frac{{p}}{{a}} \right\}\le 
  N\frac{({p}-1)({q}-1)}{{a}{b}},  \end{equation}
then $u=v=0$ a.e in $\RR^N$.} See Theorem \ref{th:mainSPC} below.\smallskip

We point out that the results proved in this paper are sharp.
To see this, one needs only to slightly  modify the examples contained in 
\cite{dm, dam-mit12}. We shall omit the tedious details.

As a concrete illustration of other results proved in this paper (see Section \ref{sec4}) we have.

Let $a\in\RR$ 
  and let $h:\RR\to ]0,+\infty[$ be a  continuous  function, then
  the problem
 \begin{equation}
   \begin{cases}-\Delta u\ge v^{a},  & on\quad\RR^3, \\ 
   \\
     -\Delta v\ge h(u) (1-\cos u), & on\quad\RR^3,\\
     \\
   u>0,\ v> 0,
 \end{cases}\end{equation}
has no non constant weak solutions. For details see Example \ref{ex:cos}.

Another example of application is the following. 

Let $\gamma,\delta \in\RR$ and let $h: \RR^+\times\RR^+\to [0,+\infty[$ be a 
  continuous nonnegative function. Then the system
 \begin{equation}
   \begin{cases}-\plapE u\ge \abs{v-1}^{\gamma},  & on\quad\RN, \\ 
   \\
     -\qlapE v\ge v^\delta+h(u,v), & on\quad\RN,\\
     \\
   u>0,\ v> 0,
 \end{cases}\end{equation}
 has no weak solutions. See Example \ref{ex:h} for a generalized version and details.

\medskip

The paper is organized as follows. In Section \ref{sec2} we give some  definitions and present few preliminary results focusing on the weak Harnack inequality and some of its consequences. Section \ref{sec3} is  devoted to the general {\it a priori} estimates for weak solutions of problem $(P)$, while in 
Section \ref{sec4} we prove our main results concerning the non--existence of non--trivial solutions of $(P)$.
Section \ref{sec5} contains some indications on 
 some extensions of the results obtained in the preceding sections to general classes of quasilinear differential operators. Finally in Appendix \ref{sec6} we quote some well known facts on Carnot groups.  

\section{Preliminaries}\label{sec2}
\bigskip

Throughout this paper we shall use some concepts briefly described in the Appendix \ref{sec6}.
For further details related to Carnot groups the interested reader may refer to \cite{bon-lan-ugu07}.

To begin with let us fix a homogeneous norm $S.$ For $R>0$,
we consider $B_R$  the ball of radius $R>0$ generated by $S$, i.e.
$B_R\decl \{x : S(x)<R\}.$  We shall also denote by  $A_R$ the annulus
$B_{2R}\setminus\overline{B_R}$.
By using the dilation $\delta_R$ and the fact that the Jacobian
of $\delta_R$ is $R^Q$, we have
\[ |B_R|=\int_{B_R} dx=R^Q\int_{B_1}dx=w_SR^Q\quad
\mathrm{and}\quad|A_R|=w_S(2^Q-1)R^Q,\]
where $w_S$ is the Lebesgue measure of the unit ball $B_1$ in $\mathbb R^N$.

Let  $p>1$ and denote by $W^{1,p}_{L,{loc}}$ the space
$$W^{1,p}_{L,{loc}}(\Omega):=\left\{u\in L^p_{{loc}}(\Omega)\,:\, |\grl u|\in L^p_{{loc}}(\Omega)\right\}.$$

Consider the system of inequalities
\begin{equation}\label{P0}
\begin{cases}-\plap u \ge f(x,u,v) \quad&\mbox{on }\Omega,\\
\\
-\qlap v \ge g(x,u,v) \quad&\mbox{on }\Omega,
\end{cases}
\end{equation}
where $\Omega\subseteq \mathbb R^N$ is an open set, 
and $f, g:\Omega\times[0,\infty)\times [0,\infty)\to[0,\infty)$  are Carath\'edory functions.

\begin{definition}
  A pair of functions 
  $(u,v)\in W^{1,p}_{L,{loc}}(\Omega)\times W^{1,q}_{L, {loc}}(\Omega)$ is a 
  {\it weak solution of} \eqref{P0} if
  $f(\cdot,u,v)$, $g(\cdot,u,v)\in L^1_{loc}(\Omega)$, 
  and the following inequalities hold
  \begin{equation}\label{ws1}
    \int_{\Omega} \abs{\grl u}^{p-2}\grl u\cdot\grl \phi_1\ge\int_{\Omega} f(x,u,v)\phi_1 
  \end{equation}
  \begin{equation}\label{ws2}
    \int_{\Omega} \abs{\grl v}^{q-2}\grl v\cdot\grl \phi_2\ge\int_{\Omega} g(x,u,v)\phi_2
  \end{equation}
  for all non--negative functions $\phi_1$, $\phi_2\in \Cuno_0(\Omega)$.
\end{definition}

\begin{lemma}[Weak Harnack inequality \cite{cdg,s,t,tw}]\label{lewh} 
  Let $Q>p>1$.
  If $u\in W^{1,p}_{L, loc}(\mathbb R^N)$ is a weak solution of
$$\begin{cases} -\plap u \ge 0 \quad&\mbox{\rm on }\Omega,\\
\\
u\ge0\quad&\mbox{\rm on }\Omega,\end{cases}$$
 then for any $\sigma\in\left(0,\frac{Q(p-1)}{Q-p}\right)$ there exists a constant $c_H>0$ independent of $u$ such that for all $R>0$
$$\left(\frac{1}{|B_R|}\int_{B_R}u^\sigma\right)^{1/\sigma}\le c_H\, \ei_{B_{R/2}}u. \eqno{\mbox{\rm(WH)}}$$
\end{lemma}

Motivated by the above results, as in \cite{dm}, we introduce the following definition.

\begin{remark}\label{consWH} Inequality (WH) implies immediately that if
    $u\in W^{1,p}_{L,loc}(\RN)$ is a weak solution of
$$\begin{cases} -\plap u \ge 0 \quad&\mbox{\rm on }\RN,\\ \\
u\ge0\quad&\mbox{\rm on }\RN,\end{cases}$$
 then
 either $u\equiv 0$ or $u>0$ a.e. on $\RN$.
 Therefore, without loss of generality we shall limit to study positive solutions.
\end{remark}

\begin{lemma}[Lemma 3.1 of \cite{dm}]\label{le3.1} Let $u:\mathbb R^N\to [0,\infty)$ be a measurable function such that $\ei_{\mathbb R^N}u=0$. Assume that {\rm (WH)} holds with exponent $\sigma>0$, then for all $\varepsilon>0$
$$\lim_{R\to\infty}\frac{|A_{R/2}\cap T_\varepsilon^u|}{|A_{R/2}|}=1,\qquad \lim_{R\to\infty}\frac{|B_R\cap T_\varepsilon^u|}{|B_R|}=1,$$
where $T_\varepsilon^u=\{x\in\mathbb R^N\, :\, u(x)<\varepsilon\}$.
\end{lemma}

Finally, few words on the hypothesis   $Q>p$. This assumption
is quite natural since the following holds.
\begin{theorem}\label{parabolic}  Let  $u\in W^{1,p}_{L,{loc}}(\RN)$ is a weak solution of
  $$\begin{cases} -\plap u \ge 0 \quad&\mbox{\rm on }\RN,\\ \\
u\ge0\quad&\mbox{\rm on }\RN.\end{cases}$$
  If $p\ge Q$, then $u$ is constant.
\end{theorem}
See  \cite{mp} for earlier results of this nature, and  \cite{dam09} for the proof and related theorems in the Carnot group setting.

\section{General a priori estimates}\label{sec3}
In this section we give some {\it a priori} bounds for the solutions of the 
system of inequalities
\begin{equation}\label{P}
\begin{cases}-\plap u \ge f(x,u,v) \quad&\mbox{on }\Omega,\\
\\
-\qlap v \ge g(x,u,v) \quad&\mbox{on }\Omega,\\
\\
u\ge 0, \; v\ge 0\quad&\mbox{on } \Omega,
\end{cases}
\end{equation}
where, we remind, $\Omega\subseteq \mathbb R^N$ is an open set
and $f, g:\Omega\times[0,\infty)\times [0,\infty)\to[0,\infty)$ are 
nonnegative Carath\'edory functions.

\begin{theorem}\label{ape} Let $(u,v)$ be a weak solution of \eqref{P}. 
   Then for all test functions $\phi_1$, $\phi_2$, every $\ell\ge0$ and 
   every $\alpha, \beta<0$, we get
   \begin{equation}\label{19}\begin{gathered}
       \int_\Omega f(x,u,v) u^\alpha_\ell\phi_1 
       +c_1\int_\Omega \abs{\grl u}^p u_\ell^{\alpha-1}\phi_1\le c_2\int_\Omega u_\ell^{\alpha-1+p}\frac{|\grl \phi_1|^p}{\phi_1^{p-1}},\\
       \int_\Omega g(x,u,v) v_\ell^\beta\phi_2 +\tilde c_1\int_\Omega
       \abs{\grl v}^q  v_\ell^{\beta-1}\phi_2\le\tilde c_2\int_\Omega v_\ell^{\beta-1+q}\frac{|\grl \phi_2|^q}{\phi_2^{q-1}},
\end{gathered}\end{equation}
where $c_1:=|\alpha|-\eta^{p'}/p'$, $c_2:=\eta^{-p}/p$, $\eta>0$, $\tilde c_1:=|\beta|-\mu^{q'}/q'$, $\tilde c_2:=\mu^{-q}/q$, $\mu>0$, $u_\ell:=u+\ell$ and $v:=v+\ell$.

If $\eta,\mu$ are so small that $c_1, \tilde c_1>0$, then for all $\alpha, \beta<0$ and $\ell\ge0$
\begin{equation}\label{22}\begin{gathered}
\int_\Omega f(x,u,v) \phi_1 \le c_3\left(\int_\Omega u_\ell^{\alpha-1+p}\frac{|\grl \phi_1|^p}{\phi_1^{p-1}}\right)^{1/p'}\left(\int_\Omega u_\ell^{(1-\alpha)(p-1)}\frac{|\grl \phi_1|^p}{\phi_1^{p-1}}\right)^{1/p},\\
\int_\Omega g(x,u,v) \phi_2 \le \tilde c_3\left(\int_\Omega v_\ell^{\beta-1+q}\frac{|\grl \phi_2|^q}{\phi_2^{q-1}}\right)^{1/q'}\left(\int_\Omega v_\ell^{(1-\beta)(q-1)}\frac{|\grl \phi_2|^q}{\phi_2^{q-1}}\right)^{1/q},
\end{gathered}\end{equation}
where $c_3:=(c_2/c_1)^{1/p'}$ and $\tilde c_3:=(\tilde c_2/\tilde c_1)^{1/q'}$.

If $u^{\alpha-1+p},u^{(1-\alpha)(p-1)}\in L^1_{loc}(A_R)$, $v^{\beta-1+q},v^{(1-\beta)(q-1)}\in L^1_{loc}(A_R)$, with $R>0$ such that $B_{2R}\Subset\Omega$, then for all $\alpha, \beta<0$ there exist $c_4, \tilde c_4>0$ for which
\begin{equation}\label{23}\begin{gathered}
\frac1{|B_R|}\int_{B_R} f(x,u,v) \le c_4 R^{-p}\left(\frac1{|A_R|}\int_{A_R} u^{\alpha-1+p}\right)^{1/p'}\left(\frac1{|A_R|}\int_{A_R} u^{(1-\alpha)(p-1)}\right)^{1/p},\\
\frac1{|B_R|}\int_{B_R} g(x,u,v) \le \tilde c_4 R^{-q}\left(\frac1{|A_R|}\int_{A_R} v^{\beta-1+q}\right)^{1/q'}\left(\frac1{|A_R|}\int_{A_R} v^{(1-\beta)(q-1)}\right)^{1/q}.
\end{gathered}\end{equation}
If there exist $\sigma>p-1$, $\delta>q-1$ such that $u^{\sigma}, v^\delta\in L^1_{loc}(\Omega)$, then
\begin{equation}\label{25}\begin{gathered}
\frac1{|B_R|}\int_{B_R} f(x,u,v) \le c_4 R^{-p}\left(\frac1{|A_R|}\int_{A_R} u^{\sigma}\right)^{(p-1)/\sigma},\\
\frac1{|B_R|}\int_{B_R} g(x,u,v) \le \tilde c_4 R^{-q}\left(\frac1{|A_R|}\int_{A_R} v^{\delta}\right)^{(q-1)/\delta}.
\end{gathered}\end{equation}
In particular, if {\rm(WH)} holds with exponent $\sigma$ for $u$ and with exponent $\delta$ for $v$, then the following inequalities hold for some appropriate constants $c_5,\tilde c_5>0$
\begin{equation}\label{27}\begin{gathered}
\frac1{|B_R|}\int_{B_R} f(x,u,v) \le c_5 R^{-p}\left(\ei_{B_R} u\right)^{p-1},\\
\frac1{|B_R|}\int_{B_R} g(x,u,v) \le \tilde c_5 R^{-q}\left(\ei_{B_R} v\right)^{q-1}.
\end{gathered}\end{equation}
\end{theorem}
\begin{proof}
  We shall prove the first inequality of (\ref{19}), (\ref{22}), (\ref{23}),
  (\ref{25}) and (\ref{27}). The remaining inequalities will follow similarly.

  Let   $\phi_1\in\Cuno_0(\Omega)$ be a nonnegative test function and set
  $r:=\mathrm{dist}(\mathrm{supp}(\phi_1),\partial\Omega)$, 
  $\Omega_r:=\{y\in\Omega \,|\, \mathrm{dist}(y,\partial\Omega)>r\}$. 
  For $\varepsilon\in(0,r)$ and $\ell>0$ we define
  $$w_\varepsilon(x):=\begin{cases}\ell+\int_{\Omega_r}D_\varepsilon(x-y)u(y)dy,\quad&\mbox{if } x\in\Omega_r,\\
0,\quad&\mbox{if } x\in\Omega\setminus\Omega_r,\end{cases}$$
  where $(D_\varepsilon)_\varepsilon$ is a family of mollifiers. See \cite{bon-lan-ugu07,dam-mit12}. 
  Thus, choosing $w_\varepsilon^{\alpha}\phi_1$ as test function in 
  \eqref{ws1} we have
  $$\int_\Omega f(x,u,v)w_\varepsilon^\alpha \phi_1+|\alpha|\int_\Omega
  \abs{\grl u}^{p-2}\grl u\cdot \grl w_\varepsilon\, w_\varepsilon^{\alpha-1}\phi_1\le\int_\Omega \abs{\grl u}^{p-1}\ |\grl\phi_1| w_\varepsilon^\alpha.$$
  Since $w_\varepsilon\to u_\ell$, $\grl w_\varepsilon\to\grl u$ in 
  $L^p_{{loc}}(\Omega_r)$ as $\varepsilon\to0$, by 
  Lebesgue's dominated convergence theorem and by duality, we get
  $$\begin{aligned}\int_\Omega f(x,u,v)u_\ell^\alpha \phi_1+|\alpha|\int_\Omega&
    \abs{\grl u}^{p} u_\ell^{\alpha-1}\phi_1\le\int_\Omega \abs{\grl u}^{p-1}
    |\grl\phi_1| u_\ell^\alpha\\
&=\int_\Omega \abs{\grl u}^{p-1}  u_\ell^{(\alpha-1)/p'}\phi_1^{1/p'}\cdot u_\ell^{(\alpha-1+p)/p}|\grl\phi_1|\phi_1^{-1/p'}\\
&\le\frac{\eta^{p'}}{p'}\int_\Omega \abs{\grl u}^{p}
u_\ell^{\alpha-1}\phi_1+\frac{1}{\eta^p p}\int_\Omega u_\ell^{\alpha-1+p}|\grl\phi_1|^p\phi_1^{1-p},
\end{aligned}$$
  where in the last step we have used the Young's inequalities.
  This completes the proof of the first inequality in \eqref{19} when $\ell>0$. 
  The case $\ell=0$ follows immediately from the case $\ell>0$, 
  by an application of Beppo--Levi's theorem letting $\ell\to0$.

  In order to prove the first inequality in \eqref{22}, 
  we use $\phi_1$ as test function in \eqref{ws1}. Let $\ell>0$, 
  by H\"older's inequality with exponent $p$ and \eqref{19} we obtain
  $$\begin{aligned}\int_\Omega f(x,u,v)\phi_1&\le\int_\Omega 
    \abs{\grl u}^{p-1} \abs{\grl \phi_1}\\
    &\le \left(\int_\Omega  \abs{\grl u}^{p} u_\ell^{(\alpha-1)}\phi_1\right)^{1/p'}\!\!\!\cdot\!\left(\int_\Omega u_\ell^{(1-\alpha)(p-1)}|\grl \phi_1|^p \phi_1^{1-p}\right)^{1/p}\\
&\le c_3 \left(\int_\Omega  u_\ell^{\alpha-1+p}|\grl \phi_1|^p \phi_1^{1-p}\right)^{1/p'}\cdot\left(\int_\Omega u_\ell^{(1-\alpha)(p-1)}|\grl \phi_1|^p \phi_1^{1-p}\right)^{1/p},\end{aligned}$$
which is the claim for $\ell>0$.  An application of Beppo--Levi's monotone convergence theorem implies the validity also for $\ell=0$.

Let $\phi_0\in \Cuno_0(\mathbb R)$ be such that $0\le\phi_0\le1$, $c_{\phi_0}:=\| |\phi'_0|^p/\phi_0^{p-1}\|_\infty<\infty$ and
$$\phi_0(t)=\begin{cases}1, \quad &\mbox{if }|t|<1,\\
0, \quad &\mbox{if }|t|>2.\end{cases}$$
  Define $\phi_1(x):=\phi_0(S(\delta_{1/R}x))$, so that
  $$\frac{|\grl \phi_1(x)|^p}{\phi_1(x)^{p-1}}=\frac{|\phi_0'(S(\delta_{1/R}x))|^p}{\phi_0^{p-1}(S(\delta_{1/R}x) )}\abs{\grl S}(\delta_{1/R}x)R^{-p}\le 
  c_{\phi_0}\norm{\grl S}_\infty R^{-p}\le cR^{-p}.$$
Hence, using $\phi_1$ as test function in \eqref{22} with $\ell=0$, we get
$$\int_\Omega f(x,u,v)\phi_1\le c_3\left(\int_{A_R} u^{\alpha-1+p}cR^{-p}\right)^{1/p'}\left(\int_{A_R} u^{(1-\alpha)(p-1)}cR^{-p}\right)^{1/p},
$$
and so, being $|A_R|=w_{S}(2^Q-1)R^Q=(2^Q-1)|B_R|$, we have
$$\begin{aligned}\frac1{|B_R|}\int_{B_R} f(x,u,v)&\le c_3(2^Q-1)cR^{-p}\left(\frac 1{|A_R|}\int_{A_R} u^{\alpha-1+p}\right)^{1/p'}\left(\frac 1{|A_R|}\int_{A_R} u^{(1-\alpha)(p-1)}\right)^{1/p},
\end{aligned}$$
which gives \eqref{23}, with $c_4:=c_3(2^Q-1)c_{\phi_0}\norm{\grl S}_\infty $.

Estimate \eqref{25} easily follows from \eqref{23}, by applying H\"older's inequality.

Finally, since (WH) holds, by \eqref{25} we obtain
\begin{eqnarray*}
\frac1{|B_R|}\int_{B_R}f(x,u,v)&\le& c_4 \left(1-\frac1{2^Q}\right)^{(1-p)/\sigma} R^{-p} \left(\frac1{|B_{2R}|}\int_{B_{2R}}u^\sigma\right)^{(p-1)/\sigma}\\
&\le& c_5 R^{-p}\left(\ei_{B_R}u\right)^{p-1},\end{eqnarray*}
with $c_5:=c_4 \left(1-\frac1{2^Q}\right)^{(1-p)/\sigma} c_H^{p-1}$. 
  Which is the first inequality in (\ref{27}) and the proof is complete.
\end{proof}\hfill$\Box$

\section{Some Liouville Theorems}\label{sec4}
In what follows $f,g:\RN\times[0,+\infty[\times [0,+\infty[ \to[0,+\infty[$ 
are supposed to be  Charatheodory functions. Let $Q>p,q>1$ and consider 
the problem, 
\begin{equation}
 \begin{cases}-\plap u\ge f(x,u,v),  \quad on\quad\RN \\ 
 \\
   -\qlap v\ge g(x,u,v), \quad on\quad\RN\\
   \\
   u>0,\ v> 0. 
 \end{cases}\label{sys1}\end{equation}
 
 Our first result  is the following.
\begin{theorem}\label{th:pos}
  Let $f(x,u,v)=f(v)$ with $f:[0,+\infty[\to ]0,+\infty[$ be continuous.
  No extra assumption on $g$.
  Then (\ref{sys1}) has no weak solutions.
\end{theorem}
 \bp If $\inf_{t\ge 0} f(t)>0$, then the non existence of solutions  is a consequence of
  Theorem~4.5 with $q=0$ in \cite{dam09}.

  Assume that $\inf_{t\ge 0} f(t)=0$. 
  From (\ref{27}) we have that for any $R>0$
  $$ R^{-p}\left(\ei_{B_R} u\right)^{p-1}\ge c \frac1{|B_R|}\int_{B_R} f(v) .$$
  Now let $0<\sigma<\frac{Q(q-1)}{Q-q}$, $R_0>0$ sufficently large 
  and set $h(t)\decl f(t^{1/\sigma})$
  and $M\decl (\ei_{B_{R_0}} u)^{p-1}$. For $R>R_0$, we have
  \begin{equation}\label{eq:tec2}
        R^{-p} M \ge c \frac1{|B_R|}\int_{B_R} h(v^\sigma) .
  \end{equation}

  Since $h$ is continuos and positive on $[0,+\infty[$, then
  there exists a positive convex nonincreasing function $h^*$ such that 
  $h(t)\ge h^*(t)>0$ and such that $h(t)\to 0$ as $t\to +\infty$
  (for an explicit construction of $h^*,$ see  \cite{CDMsteklov}).

  Therefore we obtain
   $$ C R^{-p}\ge  \mint_{B_R} h(u^\sigma) \ge\mint_{B_R} h^*(v^\sigma)
  \ge h^*\left( \mint_{B_R} v^\sigma\right). $$
  Letting $R\to +\infty$ in the last inequality we get 
  $$ h^*\left( \mint_{B_R} v^\sigma\right)\to 0\quad as\ R\to+\infty.$$
  Therefore, by construction of $h^*$, we have
   $$ \mint_{B_R} v^\sigma \to +\infty\quad as\ R\to+\infty.$$
   Since $0<\sigma<\frac{Q(q-1)}{Q-q}$ an application of  Harnack inequality, implies
   $$\ei_{B_R} v \to +\infty\quad as\ R\to+\infty.$$
  This contradiction completes the proof.
\ep

\begin{corollary}
  Let  $f(x,v,u)=f(v)$,  $g(x,v,u)=g(u)$, with 
  $f,g:]0,+\infty[\to ]0,+\infty[$ be continuous. 
  If $(u,v)$ is a weak solution of (\ref{sys1}), then necessarily 
  \begin{equation}    \label{eq:liminf}
    \liminf_{t\to 0} g(t)=0= \liminf_{t\to 0} f(t)
  \end{equation}
  and $\ei_{\RR^N} u=\ei_{\RR^N}  v=0$.
\end{corollary}
\bp
  From Theorem \ref{th:pos} we easily deduce that (\ref{eq:liminf}) holds.

  Set $m \decl \ei_{\RR^N} v$. We shall argue by contradiction. Let us  assume that $m>0$.
  Set $$v_1\decl v-m/2$$ and $$f_1(t)\decl f(t+m/2).$$ 
Clearly  $(u,v_1)$ is a weak soluton of
  \begin{equation}
 \begin{cases}-\plap u\ge f_1(v_1),  \quad on\quad\RN \\
  \\
   -\qlap{v_1}\ge g(u), \quad on\quad\RN\\
   \\
   u>0,\ v_1> 0. 
 \end{cases}\end{equation}
  Since $f_1$ is positive in $[0,+\infty[$, from Theorem \ref{th:pos}
  we reach a contradiction.
  
  Similarly, we deduce that $\ei_{\RR^N} u=0.$ The proof is complete.  
\ep

 \begin{theorem}\label{th:zero} 
   Let $f(x,v,u)=f(v)$ with $f:[0,+\infty[\to [0,+\infty[$ be continuous.
   No extra assumption on $g$. 
%
 
   Let $(u,v)$ be a weak solution of (\ref{sys1}) and let $\alpha\decl\ei_{\RR^N} v$.
   Then, $f(\alpha) =0$.
 \end{theorem}
\bp 
%
  Since the differential operator appearing in (\ref{sys1})  is translation invariant, by replacing $f$ with 
  $f(\cdot+\alpha)$) we shall assume that $\alpha =0$.

  We proceed by contradiction assuming  $m\decl f(0)>0$.
By using the  same argument and notations of the proof of Theorem \ref{th:pos}
  we deduce that (\ref{eq:tec2}) holds and $h(0)=m>0$.
Now, by a standard  continuity argument it follows that there exists $\alpha_1>0$ such that
  $h(t)>m/2$ for $t\in [0,\alpha_1]$.
  
  Let $h^*$ be the continuous function defined as follows
  \begin{equation}
    \label{eq:tec22}
    h^*(t)\decl
    \begin{cases}
      \frac{m}{2\alpha_1}(\alpha_1-t) & if\ 0\le t\le \alpha_1,\\
      0 & if\ t>\alpha_1.
    \end{cases}
  \end{equation}
  By the convexity of $h^*$, 
  from  (\ref{eq:tec2}) we deduce
  $$ C R^{-p}\ge  \mint_{B_R} h(u^\sigma) \ge\mint_{B_R} h^*(v^\sigma)
  \ge h^*\left( \mint_{B_R} v^\sigma\right). $$
  Letting $R\to +\infty$ in the last inequality we obtain that
  $$ \lim_{R\to +\infty} h^*\left( \mint_{B_R} v^\sigma\right)=0. $$
  Therefore, taking into account the construction of $h^*$, we obtain
  $$ \liminf_{R\to +\infty}  \mint_{B_R} v^\sigma\ge \alpha_1, $$
  which, in turn, by Harnack inequality, implies that $0=\ei_{\RR^N}\ge \alpha_1$.
  This contradiction completes the proof.
\ep

\begin{corollary}\label{cor:inf}
  Let  $f(x,v,u)=f(v)$, $g(x,u,v)=g(u)$
 with $f,g:[0,+\infty[\to [0,+\infty[$ be continuous functions.
 Let $(u,v)$ be a solution of  (\ref{sys1}). 
Then $\alpha:=\ei_{\RR^N} v$ is a zero of  the function  $f$ and
  $\beta:=\ei_{\RR^N} u$ is a zero of the function $g$.
\end{corollary}

\begin{theorem}\label{th:est} 
  Let  $f(x,v,u)=f(v)$, with $f:[0,+\infty[\to ]0,+\infty[$ be a continuous
  function satisfying
  $$ \liminf_{t\to 0}\frac{f(t)}{t^{{a}}}>0\quad (possibly\ +\infty), \ \ with \quad {a}>0.\eqno{(f_0)}$$
  Let  $(u,v)$ be a weak solution of (\ref{sys1}) such that ${\ei_{\RR^N}v}=0$. 
  Then there exists $c>0$ such that
  for $R$ sufficiently large, the following estimates hold

  \begin{equation}    \label{eq:i}
    {\ei_{B_R}v}\leq c R^{-p/{a}} ({\ei_{B_R}u})^{\frac{p-1}{{a}}}, 
  \end{equation}
  \begin{equation}  \label{eq:ii}
  {\int_{B_R}g(x,u,v)}\le c R^{Q-q-(q-1){p}/{a}}  
  ({\ei_{B_R}u})^{\frac{({p}-1)(q-1)}{{a}}},
  \end{equation}
  \begin{equation}   \label{eq:iii}
    {\int_{B_R}g(x,u,v)}\le c R^{Q-q-Q(q-1)/{a}}  
    \left( \int_{A_R/2} f(v) \right)^\frac{q-1}{{a}}.  \end{equation}
  Moreover if ${\ei_{\RR^N} u}=0$, and $g(x,u,v)=g(u)$ with  
  $g:[0,+\infty[\to ]0,+\infty[$ a continuous function satisfying
  $$\liminf_{t\to 0}\frac{g(t)}{t^{{b}}}>0 \quad (possibly \ +\infty),\quad with\ b>0, \eqno{(g_0)}$$
  then we have
  \begin{equation}    \label{eq:iv}
    \left({\ei_{B_R}u}\right)^{{a}{b}-({p}-1)({q}-1)}\leq c 
    R^{-{a}{q}-{p}({q}-1)}, \quad
  \left({\ei_{B_R}v}\right)^{{a}{b}-({p}-1)({q}-1)}\leq c R^{-{b}{p}-{q}({p}-1)}, 
  \end{equation}
  \begin{equation}   \label{eq:v}
    {\int_{B_R}f(v)dx\leq c
    R^{Q-{p}-({p}-1)\frac{{q}}{{b}}-Q\frac{({p}-1)({q}-1)}{{a}{b}}}\left(\int_{A_{R/2}}f(v)dx\right)^{\frac{({p}-1)({q}-1)}{{a}{b}}}.}
  \end{equation}
\end{theorem}

\bp From  ($f_0$), we deduce that there exist $c_f>0$ and  
  $\epsilon>0$
  such that
  $$f(t)\ge c_f t^{{a}} \ \ for \ \ 0<t<\epsilon. $$ 
  Set $T_\varepsilon^v\decl \{x\in\mathbb R^N\, :\, v(x)<\epsilon\}$.
  From the first inequality of (\ref{27}), we have 
  $$c_5 R^{-p}\left(\ei_{B_R} u\right)^{p-1} \ge \frac1{|B_R|}\int_{B_R} f(v)
  \ge \frac1{|B_R|}\int_{B_R\cap T_\epsilon^v} c_f v^{{a}}\ge 
  \frac{\abs{B_R\cap T_\epsilon^v}}{|B_R|}c_f (\ei_{B_R\cap T_\epsilon^v} v)^{{a}}.$$
  Next, since $\ei_{B_R\cap T_\epsilon^v} v \ge\ei_{B_R} v $, from Lemma \ref{le3.1}
  we obtain (\ref{eq:i}).

  Combining the second inequality in (\ref{27}) and (\ref{eq:i}) we deduce (\ref{eq:ii}).
  \noindent
Now,  in order to show that (\ref{eq:iii}) holds, we shall argue as follows.
  Form (\ref{27}) we have
  \begin{eqnarray*}
    \frac1{|B_R|}\int_{B_R} g(x,u,v)\le \tilde c_5 R^{-q}\left(\ei_{B_R}v\right)^{q-1}
    \\
    \le  c_5 R^{-q}\left( \frac1{\abs{A_{R/2}\cap T_\epsilon^v}}\int_{\abs{A_{R/2}\cap T_\epsilon^v}} (\ei_{A_{R/2}\cap T_\epsilon^v} v)^{{a}}\right)^{\frac{q-1}{a}}\\
    \le  
    c_5 R^{-q} \frac1{\abs{A_{R/2}\cap T_\epsilon^v}^{\frac{q-1}{a}}}\left(\frac1{c_f}\int_{\abs{A_{R/2}\cap T_\epsilon^v}}
    f(v)\right)^{\frac{q-1}{a}}\le \\
    \le c_5 R^{-q} \frac1{\abs{A_{R/2}\cap T_\epsilon^v}^{\frac{q-1}{a}}}\left(\frac1{c_f}\int_{\abs{A_{R/2}}} f(v) \right)^{\frac{q-1}{a}},
\end{eqnarray*}
  which, by  Lemma \ref{le3.1}, implies the claim.

  Assume that ($g_0$) holds. Then, from the first part of the theorem,
  for $R$ large, it follows that
  $${\ei_{B_R}u}\leq c R^{-q/{b}} ({\ei_{B_R}v})^{\frac{q-1}{{b}}} ,$$
  which, together with (\ref{eq:i}), implies the estimates in (\ref{eq:iv}).

  Similarly, we have that for $R$ large there holds
  $$\displaystyle{\int_{B_R}f(v)}\le c R^{Q-p-Q(p-1)/{b}}  
  \left( \int_{A_R/2} g(u) \right)^\frac{p-1}{{b}},$$
  which, combined with (\ref{eq:iii}), implies inequality (\ref{eq:v}), thereby concluding the proof.
\ep

\begin{theorem}\label{th:hyp}  
  Let  $f(x,v,u)=f(v)$, $g(x,u,v)=g(u)$
  with $f,g:]0,+\infty[\to [0,+\infty[$ be continuous
  functions satisfing  ($f_0$) and ($g_0$) respectively.
  If 
  \begin{equation}    \label{eq:hyp}
    \min\left\{ Q-{p}-({p}-1)\frac{{q}}{{b}},\ \  Q-{q}-({q}-1)\frac{{p}}{{a}} \right\}\le 
  Q\frac{({p}-1)({q}-1)}{{a}{b}}  \end{equation}
  then  (\ref{sys1}) has no weak solution  $(u,v)$ such that  $\ei_{\RR^N} u=\ei_{\RR^N} v=0$.
\end{theorem}

\begin{remark}
  Notice that condition (\ref{eq:hyp}) can be also written as
   \begin{equation}    \label{eq:hyp2}
    \max\left\{ abp+aq(p-1),abq+bp(q-1)\right\}\ge Q\left(ab-(p-1)(q-1) \right)
  \end{equation}
  and in the particular case $p=q$, it reads as
  \begin{equation}    \label{eq:hypp}
    \max\left\{ a+p-1,b+p-1\right\}\ge \frac{Q-p}{p(p-1)}\left(ab-(p-1)(q-1) \right).
  \end{equation}
  Finally, in the special case $p=q=2$ all the above conditions become
  \begin{equation}    \label{eq:hypp=2}
    \max\left\{ a+1,b+1\right\}\ge \frac{Q-2}{2}\left(ab-1  \right),
  \end{equation}
  which, in the Euclidean case is the inequality discovered in \cite{miti2}.
\end{remark}

\begin{remark}
  Notice that form (\ref{eq:hyp2}), it is evident that if $ab\le (p-1)(q-1)$,
  then the hypothesis (\ref{eq:hyp}) is satisfied.
\end{remark}

\bp Assume, by contradiction, that $(u,v)$ is a non trivial  weak solution of  (\ref{sys1})
  with  $\ei_{\RR^N} u=\ei_{\RR^N} v=0$.

  First consider the case $ab\le (p-1)(q-1)$. Clearly condition (\ref{eq:hyp}) holds.
  By letting $R\to +\infty$ in  (\ref{eq:iv}) of Theorem \ref{th:est}, 
  we reach a contradiction.

  Let $ab> (p-1)(q-1)$.
  Assume that 
  $$ Q-{p}-({p}-1)\frac{{q}}{{b}} \le  Q\frac{({p}-1)({q}-1)}{{a}{b}}.$$
  
  From (\ref{eq:v}) of Theorem \ref{th:est}, for any $R$ large we have
  \begin{equation}    \label{eq:tec1}
    \int_{B_R}f(v)dx\leq c   
\left(\int_{A_{R/2}}f(v)dx\right)^{\frac{({p}-1)({q}-1)}{{a}{b}}},  \end{equation}
  which implies that
  $$\left(\int_{B_R}f(v)dx\right)^{\frac{ab-({p}-1)({q}-1)}{{a}{b}}}
\le c.$$
  Therefore, we obtain that $f(v)\in L^1(\RN)$. Hence, from (\ref{eq:tec1}) it follows that,
  $$f(v(x)) = 0\quad \mbox{a.e. on}\quad \RN.$$ Using this information in   
  (\ref{eq:tec1}) and the condition ($f_0$),  
  for $\epsilon>0$ sufficently small,
  (small enough such that $f(t)\ge c_f t^a$ for $t\in ]0,\epsilon[$),
  we obtain
  $$c_f\int_{ T_\varepsilon^v}  v^a \le \int_{ T_\varepsilon^v }f(v)dx =0 $$  
  where $T_\epsilon^v=\{x\in\mathbb R^N\, :\, v(x)<\epsilon\}$.

  Now since $v\not\equiv 0$, by Harnack's inequality $v^a>0$ a.e. on $\RN$,
  therefore, necessarily $\abs{T_\varepsilon^v}=0$. This implies that
  $v\ge \epsilon$ a.e. contradicting  the fact that $\ei_{\RR^N} u=0$.

  If
  $$ Q-{q}-({q}-1)\frac{{p}}{{a}} \le  Q\frac{({p}-1)({q}-1)}{{a}{b}},$$
  the proof is similar.
\ep

\begin{example} Consider,
 \begin{equation}\label{ex1}
   \begin{cases}-\plap u\ge f(v),   &on\quad\RN, \\ 
   \\
   -\qlap v\ge g(x,u,v), & on\quad\RN,\\
   \\
   u>0,\ v> 0. 
 \end{cases}\end{equation}

  i) If $f(v)=v^{-\gamma}$ with $\gamma>0$ and    $g:\RN\times \RR^+\times\RR^+\to ]0,+\infty[$ is  Carath\'eodory, then the problem (\ref{ex1}) has no weak solution.

ii)   If $f(v)=\frac{1}{1+v^\gamma}$ with $\gamma>0$ and $g:\RN\times \RR^+\times\RR^+\to ]0,+\infty[$ is  Carath\'eodory, then the problem (\ref{ex1}) has no weak solution.

  Indeed, in both cases  the claim follows from Theorem \ref{th:pos}.

  Notice that we do not assume any growth  assumption on $g$.
\end{example}

\begin{example}\label{ex:cos} Let $a\in\RR$ 
  and let $h:\RR\to ]0,+\infty[$ be a  continuous  function.
   
  Consider,
 \begin{equation}
   \begin{cases}-\Delta u\ge v^{a},  & on\quad\RR^3, \\ 
   \\
     -\Delta v\ge h(u) (1-\cos u), & on\quad\RR^3,\\
     \\
   u>0,\ v> 0.
 \end{cases}\end{equation}

  The problem has no non constant weak solutions. Indeed, in the case $a\le 0$ the claim follows from the 
 previous example. Let $a>0$.
 From Corollary \ref{cor:inf} it follows that
 $\ei_{\RR^N} v=0$ and $\ei_{\RR^N} u= 2k\pi$ where  $k$ is an integer.
 By traslation invariance we can assume that $k=0$. Now we are in the position to apply
 Theorem \ref{th:hyp}. In this case $b=2$ and the hypothesis (\ref{eq:hyp}), 
 or equivalently (\ref{eq:hypp=2}), is satisfied provided
 $$\max \{a+1,3\}\ge \frac12 (2a-1)=a-\frac12.$$
Notice that the above inequality holds for any $a>0$.
\end{example}

\begin{example}\label{ex:h}
Let $\gamma,\delta \in\RR$ and let $h: \RR^+\times\RR^+\to [0,+\infty[$ be a 
  continuous nonnegative function. Consider the system
 \begin{equation}
   \begin{cases}-\plap u\ge \abs{v-1}^{\gamma},  & on\quad\RN, \\ 
   \\
     -\qlap v\ge v^\delta+h(u,v), & on\quad\RN,\\
     \\
   u>0,\ v> 0.
 \end{cases}\end{equation}

  Our claim is that this problem has no weak solutions. 
  
  Indeed, in the case $\gamma\le 0$ or $\delta\le 0,$ this follows by
  applying Theorem \ref{th:pos} or Corollary 2.4 of \cite{dm} respectively.
  
 Consider now the case $\gamma,\delta>0$.
  From Theorem  \ref{th:zero} it follows that $\ei_{\RR^N} v=1.$
  
  Using this information, and setting $v_1\decl v-1$, we see that $(u,v_1)$ is a weak solution of
   \begin{equation}
   \begin{cases}-\plap u\ge v_1^{\gamma},  & on\quad\RN, \\ 
   \\
     -\qlap {v_1}\ge 1, & on\quad\RN,\\
     \\
   u>0,\ v_1> 0.
 \end{cases}\end{equation}
In other words $(u,v_1)$ is a weak solution of (\ref{sys1}) with $g=g(u)=1>0$. 
 An application of Theorem \ref{th:pos} implies the claim.
\end{example}

\begin{example} Let $\gamma,\delta>0 $ and let $h:\RR^+\to ]0,+\infty[$ be 
  a continuous positve function. 
  By an  argument  similar to the one used in the above  example, 
 we can  show that the system
 \begin{equation}
   \begin{cases}-\plap u\ge \abs{v-1}^{\gamma},  & on\quad\RN, \\ 
   \\
     -\qlap v\ge v^\delta h(u), & on\quad\RN,\\
     \\
   u>0,\ v> 0,
 \end{cases}\end{equation}
 has no weak solutions. We omit the details.
\end{example}

\section{Some Extensions}\label{sec5}
 In what follows we suppose that $f,g:\RN\times]0,+\infty[\times ]0,+\infty[\to[0,+\infty[$
are two nonnegative Charatheodory functions.

Let $p_1,p_2>1$ and  for $i=1,2$, 
$\A_{p_i}: \RN\times\mathbb R\times\mathbb R^l\to\mathbb R^l$ denotes
a Carath\'eodory function.
We assume that the function $\mathscr A_{p_i}$ is W-$p_i$-C, {\it weakly-$p_i$-coercive}, namely  there exists a constant $k>0$ such that
$$(\mathscr A_{p_i}(x,t,\xi)\cdot \xi)\ge k|\mathscr A_{p_i}(x,t,\xi)|^{p_i'}\quad\mbox{for all } (x,t,\xi)\in\RN \times\mathbb R\times\mathbb R^l. \eqno(\mbox{W-}p_i\mbox{-C}).$$
See \cite{bvp,mp,s} for details.

Consider the  following,
\begin{equation}
 \begin{cases}-\diverl(\A_1(x,u,\grl u)\ge f(x,u,v), & \quad on\quad\RN \\ 
 \\
  -\diverl(\A_2(x,v,\grl v)\ge g(x,u,v), &\quad on\quad\RN,\\
  \\  u>0,\ v> 0. 
 \end{cases}\label{sys1w}\end{equation}

As usual, a pair of functions $(u,v)\in W^{1,p_1}_{L,loc}(\RN)\times W^{1,p_2}_{L,loc}(\RN)$
 is a {\it weak solution of} \eqref{sys1w} 
if
$f(\cdot,u,v),\ g(\cdot,u,v), \abs{A_{p_1}(\cdot,u,\grl u)}^{p_1'},
\abs{A_{p_2}(\cdot,v,\grl v)}^{p_2'} \in L^1_{loc}(\RN)$,
and 
\begin{equation}
\int_{\RN}(\mathscr A_p(x,u,\grl u)\cdot\grl \phi_1)\ge\int_{\RN} f(x,u,v)\phi_1 
\end{equation}
\begin{equation}
\int_{\RN}(\mathscr A_q(x,v,\grl v)\cdot\grl \phi_2)\ge\int_{\RN} g(x,u,v)\phi_2
\end{equation}
for all non--negative functions $\phi_1$, $\phi_2\in \Cuno_0(\RN)$.

We shall assume that $Q>p_i>1$. This restriction is justified by the fact that
an analogue of Theorem \ref{parabolic} holds. Namely, if $w$ is a weak solution of 
the problem (\ref{eq:superw}) below and $p_i>Q$, then $w$ is a constant. See \cite{dam09}.

Furthermore we shall suppose that a weak Harnack inequality (WH) holds for solutions of
\begin{equation}\label{eq:superw}
  \begin{cases}-\mathrm{div}(\mathscr A_{p_i}(x,w,\grl w))\ge0&\quad\mbox{on }\RN,\\
  \\
w\ge 0&\quad\mbox{on }\RN,
\end{cases} \qquad i=1,2.\end{equation}
Namely, for $i=1,2$ there exists $\sigma_i>p_i-1$ and $c_H$ such that 
if  $w$ is a weak solution of (\ref{eq:superw}), then for any $R>0$ we have
$$\left(\frac{1}{|B_R|}\int_{B_R}w^{\sigma_i}\right)^{1/\sigma_i}\le c_H\, \ei_{B_{R/2}}w. \eqno{\mbox{\rm(WH)}}$$

Examples of operators for which the weak Harnack inequality holds
are given by the following.
\begin{lemma}(Weak Harnack Inequality, see \cite{cdg}) Let $Q>p_i>1$. Let  $\A_{p_i}$ be 
  {\bf S-$p_i$-C} ({\it strongly-$p_i$-coercive}), that is 
  there exist two constants $k,h >0$ such that
$$(\mathscr A_{p_i}(x,t,\xi)\cdot \xi)\ge h |\xi|^{p_i}\ge k |\mathscr A_{p_i}(x,t,\xi)|^{p_i'}\quad\mbox{for all } (x,t,\xi)\in\RN\times\mathbb R\times\mathbb R^l. \eqno(\mbox{S-}p_i\mbox{-C})$$
Then for any $\sigma_i\in (0,\frac{Q(p_i-1)}{Q-p_i})$ there exists $c_{H}>0$ 
such that
for any $u\in W^{1,p_i}_{L,loc}(\RN)$  weak solution of (\ref{eq:superw}),
and $R>0$, (WH) holds.
\end{lemma}

The results of the previous Section can be reformulated for the problem (\ref{sys1w})
as follows.
\begin{theorem}
 Let $f(x,v,u)=f(v)$ with $f:[0,+\infty[\to ]0,+\infty[$ be continuous.
 No extra assumption on $g$.
 Then (\ref{sys1w}) has no solutions.
 \end{theorem}
\begin{corollary} 
 Let  $f(x,v,u)=f(v)$,  $g(x,v,u)=g(u)$, with  $f,g:]0,+\infty[\to ]0,+\infty[$ be continuous. 
 If $(u,v)$ is a solution of (\ref{sys1w}),
 then (\ref{eq:liminf}) holds and $\ei_{\RR^N} u=\ei_{\RR^N} v=0$.
\end{corollary}
\begin{theorem} Let $\A_{p_i}=\A_{p_i}(x,\xi)$ for $i=1,2$.
   Let $f(x,v,u)=f(v)$ with $f:[0,+\infty[\to [0,+\infty[$ be continuous.
   No extra assumption on $g$.
   Let $(u,v)$ be a weak solution of (\ref{sys1w}) and let $\alpha\decl\ei_{\RR^N} v$.
   Then, $f(\alpha) =0$.
 \end{theorem}

\begin{theorem}
Let  $f(x,v,u)=f(v)$, with $f:[0,+\infty[\to ]0,+\infty[$ be a continuous
  function satisfing ($f_0$).
  Let  $(u,v)$ be a weak solution of (\ref{sys1w}) such that ${\ei_{\RR^N}v}=0$. 
  Then there exists $c>0$ such that
  for $R$ sufficiently large, the following estimates hold
  \begin{eqnarray}
    {\ei_{B_R}v}\leq R^{-p_1/a} ({\ei_{B_R}u})^{\frac{p_1-1}{a}},\\
    {\int_{B_R}g(x,u,v)}\le c R^{Q-p_2-(p_2-1)p_1/a}  
      ({\ei_{B_R}u})^{\frac{(p_1-1)(p_2-1)}{a}},\\
    {\int_{B_R}g(x,u,v)}\le c R^{Q-p_2-Q(p_2-1)/a}  
      \left( \int_{A_R/2} f(v) \right)^\frac{p_2-1}{a}.
  \end{eqnarray}
Moreover if $\ei_{\RR^N} u=0$ and $g(x,u,v)=g(u)$ with  $g:[0,+\infty[\to ]0,+\infty[$ 
a continuous function satisfying  ($g_0$), then we have
\begin{eqnarray}
  \left({\ei_{B_R}u}\right)^{{ab-(p_1-1)(p_2-1)}}\leq c R^{-ap_2-p_1(p_2-1)},\\
  \left({\ei_{B_R}v}\right)^{{ab-(p_1-1)(p_2-1)}}\leq c R^{-bp_1-p_2(p_1-1)},\\
  \int_{B_R}f(v)dx\leq c
    R^{Q-p_1-(p_1-1)\frac{p_2}{b}-Q\frac{(p_1-1)(p_2-1)}{ab}}\left(\int_{A_{R/2}}f(v)dx\right)^{\frac{(p_1-1)(p_2-1)}{ab}}.
\end{eqnarray}
\end{theorem}
\begin{theorem}\label{th:mainSPC}
  Let  $f(x,v,u)=f(v)$, $g(x,u,v)=g(u)$
  with $f,g:]0,+\infty[\to [0,+\infty[$ be continuous
  satisfying ($f_0$) and ($g_0$)  respectively and assume that
  (\ref{eq:hyp}) holds.
  Then   (\ref{sys1w}) has no weak solution $(u,v)$ such that  
  $\ei_{\RR^N} u=\ei_{\RR^N} v=0$.
\end{theorem}

\begin{remark}
 As a final observation, we point out that most of the results proved in this section hold for systems associated to (W-$p$-C) operators and power nonlinearities.
 We refer the interested reader to \cite{bvp}and \cite{mp}  for the Euclidean setting,  and to \cite{dam10}  for precise formulation and interesting open problems in the Carnot group framework.
\end{remark}

\section{Appendix}\label{sec6}

We quote some  facts on Carnot groups  and refer the interested reader to  
\cite{bon-lan-ugu07,fol75,fol-ste82} 
  for  more detailed information on this subject.

A Carnot group is a connected, simply connected, nilpotent Lie
group $\G$ of dimension $N$ with graded Lie algebra ${\cal
G}=V_1\oplus \dots \oplus V_r$ such that $[V_1,V_i]=V_{i+1}$ for
$i=1\dots r-1$ and $[V_1,V_r]=0$. Such an integer $r$ is called the
\emph{step} of the group.
 We set $l=n_1=\dim V_1$, $n_2=\dim V_2,\dots,n_r=\dim V_r$.
A  Carnot group $\G$ of dimension $N$ can be identified, up to an
isomorphism, with the structure of a \emph{homogeneous Carnot
Group} $(\RN,\circ,\delta_R)$ defined as follows; we
identify $\G$ with $\RN$ endowed with a Lie group law $\circ$. We
consider $\RN$ split in $r$ subspaces
$\RN=\RR^{n_1}\times\RR^{n_2}\times\cdots\times\RR^{n_r}$ with
$n_1+n_2+\cdots+n_r=N$ and $\xi=(\xi^{(1)},\dots,\xi^{(r)})$ with
$\xi^{(i)}\in\RR^{n_i}$. 
We shall assume that for any $R>0$ the dilation
$\delta_R(\xi)=(R\xi^{(1)},R^2 \xi^{(2)},\dots,R^r \xi^{(r)})$ 
is a Lie group automorphism.
The Lie algebra of
left-invariant vector fields on $(\RN,\circ)$ is $\cal G$. For
$i=1,\dots,n_1=l $ let $X_i$ be the unique vector field in $\cal
G$ that coincides with $\partial/\partial\xi^{(1)}_i$ at the
origin. We require that the Lie algebra generated by
$X_1,\dots,X_{l}$ is the whole $\cal G$.

We denote with $\grl$ the vector field $\grl\decl(X_1,\dots,X_l)^T$
and we call it \emph{horizontal vector field} and by  $\diverl$ the formal 
adjont on $\grl$, that is (\ref{eq:defdiv}).
Moreover, the vector
fields $X_1,\dots,X_{l}$ are homogeneous of degree 1 with respect
to $\delta_R$ and in this case
$Q= \sum_{i=1}^r i\,n_i=  \sum_{i=1}^r i\,\mathrm{dim}V_i$ is called the 
\emph{homogeneous dimension} of $\G$.
The \emph{canonical sub-Laplacian} on $\G$ is the
second order differential operator defined by
$$\Delta_G=\sum_{i=1}^{l} X_i^2=\diverl(\grl\cdot )$$ and for $p>1$ the
$p$-sub-Laplacian operator is
$$\Delta_{G,p} u\decl \sum_{i=1}^{l} X_i(\abs{\grl u}^{p-2}X_iu)=
   \plap{u}.$$
Since $X_1,\dots,X_{l}$
generate the whole $\cal G$, the sub-Laplacian $\Delta_G$ satisfies the
H\"ormander hypoellipticity  condition. 

In this paper $\nabla$ and $|\cdot|$ stand  respectively for the
usual gradient in $\RN$
and the Euclidean norm.

Let $\mu\in \C(\RR^N;\RR^l)$ be a matrix 
$\mu\decl(\mu_{ij})$, $i=1,\dots,l$, $j=1,\dots,N$.
For $i=1,\dots,l$,  let $X_i$ and its formal adjoint $X_i^*$
be defined as
\begin{equation} X_i\decl\sum_{j=1}^N \mu_{ij}(\xi)\frac{\partial}{\partial \xi_j},
\qquad X_i^*\decl-\sum_{j=1}^N \frac{\partial}{\partial\xi_j}\left(\mu_{ij}(\xi)\cdot\right),
   \label{mu}\end{equation}
and let $\grl$ be the vector field defined by
$\grl\decl (X_1,\dots,X_l)^T=\mu\nabla$
and $\grl^*\decl(X_1^*,\dots,X_l^*)^T$.

For any vector field $h=(h_1,\dots,h_l)^T\in\Cuno(\Omega,\RR^l)$, we shall use
the following notation $ \diverl(h)\decl\diver(\mu^T h)$,
that is
\begin{equation}  \label{eq:defdiv}
   \diverl(h)=-\sum_{i=1}^l X_i^*h_i=-\grl^*\cdot h.\end{equation}

An assumption that we shall made (which actually is an assumption on the matrix
$ \mu$) is that the operator
$$  \Delta_Gu =\diverl(\grl u)$$
is a canonical sub-Laplacian on a Carnot group 
(see below for a more precise meaning).
The reader, which is not acquainted with these structures, 
can think to the special case 
of $\mu =I$, the identity matrix in $\RN$, that is the usual Laplace operator
in Euclidean setting.

A nonnegative continuous function $S:\RN\to\RR_+$ is called a 
\emph{homogeneous norm} on { $\G$}, if 
$S(\xi^{-1})=S(\xi)$, $S(\xi)=0$ if and only if $\xi=0$, and it is
homogeneous of degree 1 with respect to $\delta_R$ (i.e.
$S(\delta_R(\xi))=R S(\xi)$).
A homogeneous norm $S$ defines on $\G$ a \emph{pseudo-distance} defined as
$d(\xi,\eta)\decl S(\xi^{-1}\eta)$, which
in general is not a distance.
If $S$ and $\tilde S$ are two homogeneous norms, then they are equivalent,
that is, there exists a constant
$C>0$ such that $C^{-1}S(\xi)\le \tilde S(\xi)\le CS(\xi)$.
Let $S$ be a homogeneous norm, then there exists a constant
$C>0$ such that $C^{-1}\abs\xi\le S(\xi)\le C\abs\xi^{1/r}$,
for $S(\xi)\le1$.
An example of homogeneous norm is
$  S(\xi)\decl\left(\sum_{i=1}^r\abs{\xi_i}^{2r!/i}\right)^{1/2r!}.$

Notice that if $S$ is a homogeneous norm differentiable a.e., 
then $\abs{\grl S}$ is homogeneous of degree 0 with respect to 
$\delta_R$; hence $\abs{\grl S}$ is bounded.

We notice that in a Carnot group, the Haar measure coincides with the Lebesgue measure.

\medskip

Special examples of Carnot groups are the
  Euclidean spaces $\RR^Q$.
  Moreover, if $Q\le 3$ then any Carnot group is the ordinary Euclidean
  space $\RR^Q$.

 The simplest nontrivial example of a Carnot group
  is the Heisenberg group $\hei^1=\RR^3$.
  For an integer $n\ge1$, the Heisenberg group $\hei^n$ is defined as follows:
  let $\xi=(\xi^{(1)},\xi^{(2)})$ with
  $\xi^{(1)}\decl(x_1,\dots,x_n,y_1,\dots,y_n)$ and $\xi^{(2)}\decl t$.
  We endow $\RR^{2n+1}$ with  the group law
$\hat\xi\circ\tilde\xi\decl(\hat x+\tilde x,\hat y+\tilde y,\hat t+ \tilde t+2\sum_{i=1}^n(\tilde x_i\hat y_i-\hat x_i \tilde y_i)).$
We consider the vector fields
\[X_i\decl\frac{\partial}{\partial x_i}+2y_i\frac{\partial}{\partial t},\
        Y_i\decl\frac{\partial}{\partial y_i}-2x_i\frac{\partial}{\partial t},
        \qquad\mathrm{for\ } i=1,\dots,n, \]
and the associated  Heisenberg gradient
$ \grh\decl (X_1,\dots,X_n,Y_1,\dots,Y_n)^T$.
The Kohn Laplacian $\lh$ is then the operator defined by
$\lh\decl\sum_{i=1}^nX_i^2+Y_i^2.$
The family of dilations is given by
$\delta_R(\xi)\decl (R x,R y,R^2 t)$ with homogeneous dimension
$Q=2n+2$.
In ${\hei^n}$ a canonical homogeneous norm is defined as
$\abs{\xi}_H\decl \left(\left(\sum_{i=1}^n x_i^2+y_i^2\right)^2+t^2\right)^{1/4}.$

\subsection*{Acknowledgements}
\medskip
{This work is  supported by the Italian MIUR National Research Project: Quasilinear Elliptic Problems and Related Questions.}

\end{document}